\documentclass[reqno]{amsart}
\usepackage[margin=1.2in]{geometry}
\usepackage{graphicx} 
\usepackage{amsfonts}
\usepackage{amssymb}
\usepackage{amsthm}
\usepackage{amsmath}
\usepackage{tikz}
\usepackage{caption}
\usepackage{tcolorbox}
\usepackage{hyperref}
\usepackage{cleveref}
\usepackage{enumitem}
\usepackage{mathtools}

\newtcbox{\mybox}[1]{colback=red!5!white, colframe=red!75!black,fonttitle=\bfseries, title={#1}}

\newcommand\Item[1][]{
  \ifx\relax#1\relax  \item \else \item[#1] \fi
  \abovedisplayskip=1pt
  \abovedisplayshortskip=1pt~\vspace*{-\baselineskip}}

\newtheorem{thm}{Theorem}[section]

\newtheorem*{introthm*}{Theorem}

\newtheorem{cor}[thm]{Corollary}
\newtheorem{lemma}[thm]{Lemma}
\newtheorem{prop}[thm]{Proposition}

\newtheorem*{oprobl*}{Open Problem}

\theoremstyle{definition}

\newtheorem{ex}[thm]{Example}

\theoremstyle{remark}

\numberwithin{equation}{section}

\newcommand{\R}{\ensuremath{\mathcal{R}}}
\DeclareMathOperator{\init}{in}
\DeclareMathOperator{\lex}{lex}
\DeclareMathOperator{\Img}{Im}
\DeclareMathOperator{\coker}{Coker}

\title{Linear strands of powers of certain binomial edge ideals}
\author[Dohadwala]{Abbas Dohadwala}
\address{Department of Mathematics, Purdue University, Mathematical Sciences Bldg, 150 N University St, West Lafayette, Indiana 47907, USA.}
\email{adohadwa@purdue.edu}
\author[Flores-Silva]{Bryan Flores-Silva}
\address{Facultad de Ciencias, Universidad de Colima, 340 Bernal Díaz del Castillo St, Villas San Sebastián, Colima, Colima 28045, México.}
\email{bryanefs12@gmail.com}
\author[Orozco-Moya]{Alicia Orozco-Moya}
\address{Centro Universitario de Ciencias Exactas e Ingenierías, Universidad de Guadalajara, 1421 General Marcelino García Barragán Ave, Olímpica, Guadalajara, Jalisco 44430, México.}
\email{luz.orozco5616@alumnos.udg.mx}
\author[Siegelnickel]{Zoe Siegelnickel}
\address{Department of Mathematics and Statistics, Boston University, 665 Commonwealth Ave, Boston, Massachusetts 02215, USA.}
\email{z0e@bu.edu}

\date{\today}

\begin{document}

\maketitle

\begin{abstract}
We provide a closed formula for the graded Betti numbers in the linear strands of all powers of  binomial edge ideals $J_G$ arising from closed graphs $G$ that do not have the complete graph $K_4$ as an induced subgraph. We show that these agree with the corresponding Betti numbers for the powers of the lexicographic initial ideal of $J_G$, thereby confirming a conjecture of Ene--Rinaldo--Terai in a special case.
\end{abstract}

\section{Introduction}

Binomial edge ideals, introduced by Herzog, Hibi, Hreinsd\'ottir, Kahle, and Rauh in \cite{herzog2010binomial} and Ohtani in \cite{Ohtani}, are ideals generated by binomials corresponding to the edges of a simple graph. Specifically, the binomial edge ideal of a graph $G$ with edge set $E(G)$ is 
\[J_{G}=\left (x_iy_j - y_ix_j : \{i,j\} \in E(G)\right).
\]
These ideals have been extensively studied due to their rich algebraic structure and deep connections to combinatorial properties of the underlying graphs; see \cite[Chapter 7]{HHO}. 
In particular, the property of a labeled graph of being closed has proven fundamental in the study of binomial edge ideals. This property ensures that the minimal generators of $J_{G}$ form a Gr\"obner basis with respect to the lexicographic order. 
We denote by $\init_{\lex}$ the initial ideal with respect to this order.

    A conjecture by Ene, Rinaldo and Terai in \cite[Conjecture 6.2]{ene2016powers} asserts that for a closed graph \({G}\) the graded Betti numbers of $J_{G}^m$ and $\operatorname{in}_{lex}{J_{G}^m}$ are equal for all $m\geq 1$. In the case of the first power ($m=1$) this conjecture was recently proved  by Peeva in \cite{Peeva}. Prior work by Herzog--Kiani--Madani in \cite{HKM} and by Almousa--VandeBogert in \cite{AV} had determined the linear strands in the resolutions of $J_{G}$ and
    $\operatorname{in}_{lex}{J_{G}}$ for closed graphs $G$, showing in particular that the Betti numbers in these strands were equal. To our knowledge the linear strands of powers of binomial edge ideals have not yet been determined in the literature.
    
    In this paper, we make progress towards the Ene-Rinaldo-Terai conjecture  by providing a closed formula for the graded Betti numbers in the linear strands of all powers $J_{G}^m$ and $\operatorname{in}_{lex}{J_{G}^m}$ when $G$ is a closed graph that does not admit the complete graph on four vertices as an induced subgraph. In particular, under these hypotheses the Betti numbers in the linear strand  only depend on the number of edges and triangles of $G$.
    
    Our main result is the following.
    
    \begin{introthm*}[\Cref{thm: linear strand}] Let $G$ be closed and $K_4$-free graph with $e$ edges and $t$ triangles. Then for each $m\geq 1$ and $i\geq 0$ the following identities hold
   \[
   \beta_{i, 2m+i}(J_G^m)=\beta_{i, 2m+i}(\init_{\lex} J_G^m)=\binom{e+m-i-1}{m-i}\binom{2t}{i}.
   \] 
    \end{introthm*}

Our methods employ tools from the study of Rees algebras of determinantal facet ideals. In parti\-cu\-lar, crucial to our work are the results of Almousa--Lin--Liske in\cite{ALL}, summarized here in \Cref{thm: rees}.
    
\section{Background}

A graph \(G\) is {\bf closed} if there exists a labeling of the vertices such that for all edges $\{i,j\}$ and $\{k,l\}$ with $i<j, \ k<l$, we dhave that $\{j,l\} \in E(G)$ if $i= k$ and $\{i,k\} \in E(G)$ if $j=l$. However, there are many equivalent definitions; see \cite[section 7.1.1]{HHO}. A clique $\Delta$ of $G$ is an induced complete subgraph of $G$. We denote the $3$-cycle graph by $C_3$ and the complete graph on $4$ vertices by $K_4$.

The {\bf lexicographic order} on the polynomial ring $R=\mathbb{K}[x_1, \ldots, x_n]$ with coefficients in a field $\mathbb{K}$ induced by $x_1 > x_2 > \cdots > x_n$ is the dictionary order where 
 $x_1^{a_1}\cdots x_n^{a_n}<_{lex}x_1^{b_1}\cdots x_n^{b_n}$ provided the leftmost nonzero entry of $(a_1-b_1, \ldots, a_n-b_n)$ is negative.
For an ideal  $I\subset R$  let $\init_{\lex}(I)$ denote the ideal generated by the leading monomials of elements of $I$ with respect to the lexicographic ordering.

A graded free resolution is a sequence of free $R$-modules $F_i$ and degree-preserving $R$-module homomorphisms $d_i$ 
\begin{equation}
\label{eq:res}
0 \rightarrow F_p \xlongrightarrow{d_p}
\cdots \xlongrightarrow{d_3} F_2 \xlongrightarrow{d_2} F_1 \xlongrightarrow{d_1} F_0
 \end{equation}
such that  ${\rm Ker}(d_i)={\rm Im}(d_{i+1})$ for each $i$ with $1\leq i\leq p$. Sequence \eqref{eq:res} is termed a free resolution for the $R$-module $M=\coker(d_1)$.
A free resolution is  minimal if  $\Img(d_i)\subseteq \mathfrak{m}F_{i-1}$ for each $i$, where $\mathfrak{m}$ denotes the homogeneous maximal ideal of $R$. 

Let $R(-j)$ denote the ring $R$ with grading shifted according to the rule $[R(-j)]_i=R_{i-j}$.
The 
{\bf graded Betti number} $\beta_{ij}(M)$ is the number of times $R(-j)$ appears as a summand in the $i$-th free module of a minimal free resolution of $M$.  More formally, with notation  as in \eqref{eq:res}, we have
\[
F_i=\bigoplus_j R(-j)^{\beta_{ij}} \text{ for all }i,j.
\]
The $i$-th syzygy module of $M$ is the $R$-module $\Omega^i(M)=\ker(d_i)$, uniquely determined up to isomorphism. If $M$ is generated in degree $d$, we let $F_i^{\rm lin}$ be the direct summand $R(-d-i)^{\beta_{i,d+i}}$ of $F_i$. Since $d_i(F_i^{\rm lin})\subseteq F_{i-1}^{\rm lin}$,
the complex obtained by restricting the differential of \eqref{eq:res} to the free modules $F_i^{\rm lin}$ is a subcomplex of \eqref{eq:res}, which 
we call the {\em linear strand} of the resolution of $M$. It has the form
\[
0\rightarrow R(-d-p)^{\beta_{p,d+p}} \xlongrightarrow{d_p}
\cdots \xlongrightarrow{d_3}  R(-d-2)^{\beta_{2,d+2}} \xlongrightarrow{d_2}
R(-d-1)^{\beta_{1,d+1}}
\xlongrightarrow{d_1}
R(-d)^{\beta_{0d}}.
\]

Let $G$ be a  graph with vertex set $V=[n]$ and binomial edge ideal  $J_{G}$. Set 
    \begin{equation}\label{eq: JG}
    f_{ij}=x_iy_j-x_jy_i \text{ and }    J_G=(\ f_{ij} : \{i,j\}\in E(G)).
    \end{equation} 
The {\bf Rees algebra} of $J_G$ is the graded $R$-algebra
    \[\mathcal{R}(J_{G}):=\bigoplus_{i\geq 0} J_{G}^iT^i\subset R[T].\] 
      It admits a surjective ring homomorphism 
    \begin{equation}\label{eq: phi}
        \varphi: R[T_{ij} \mid (i,j)\in E(G)] \longrightarrow\mathcal{R}(J_G),\quad r\mapsto r, \text{ for }r\in R,\quad T_{ij}\mapsto f_{ij}T.
    \end{equation}
  The kernel of \(\varphi\) corresponds to  the relations of the Rees algebra of $J_G$. The Rees algebra is bigraded, with $\deg(T) = (1,0)$, $\deg(x_i) = \deg(y_i) = (0,1)$ and the map $\varphi$ is degree-preserving, which implies $\deg(T_{ij}) = (1,2)$.

The  Rees algebra of the initial ideal is defined analogously: $\mathcal{R}(\init_{\lex}( J_G)=\bigoplus_{i \geq 0} (\init_{\lex} J_G)^iT^i \subset R[T]$. It is shown in \cite[Equation (3)]{ene2016powers} that for a closed graph $G$  identities $\init_{\lex} (J_G^i) = (\init_{\lex} J_G)^i$ hold for all $i\geq 1$. Therefore, we have that
$\init_{\lex} (\mathcal R(J_G)) = \mathcal R(\init_{\lex}J_G)$.

When $G$ is a closed graph, Almousa-Liske–Lin showed that the relations of the Rees algebra of its binomial edge ideal and initial ideal can be described explicitly as follows.

\begin{thm}[{\cite[Theorem 4.5]{ALL}}]\label{thm: rees}
    Let $G$ be a closed graph with cliques $\Delta_1, \ldots, \Delta_s$. For $\{i,j\}\in E(G)$ write $\{i,j\}\in \Delta_a$ if $a$ is the smallest index of a cliques that contains $\{i,j\}$. Adopt the notation in \eqref{eq: JG} and \eqref{eq: phi}. In particular let $T_{ij}$ be such that $\varphi(T_{ij})=f_{ij}$ and $\varphi(T_{ij})=\init_{lex}(f_{ij})$, respectively.
    
     The relations of $\mathcal{R}(J_G)$ are generated by:
    \begin{enumerate}
        \Item \label{K}
        \[
            {x_iy_jT_{i'j'}} - x_jy_iT_{i'j'} -x_{i'}y_{j'}T_{ij} + x_{j'}y_{i'}T_{ij} \text{ for }  \{i,j\}\in\Delta_a, \{i',j'\}\in \Delta_b \]
            \text{  with } $a\neq b $  (Koszul relations);
        \Item  \begin{equation}
           {x_iT_{jk}}-x_jT_{ik}+x_kT_{ij}  \text{ and } {y_jT_{ik} }-y_iT_{jk}-y_kT_{ij} \label{EN-rels}
        \end{equation} for $\{i,j,k\}$  an induced $C_3$ in $G$ with $i<j<k$ (Eagon-Northcott relations);
        \Item \[
            {T_{ij}T_{kl}} -T_{ik}T_{jl} + T_{il}T_{jk} = 0
       \] where $1\leq i<j<k<l\leq n$ and $\{i, j, k, l\}$  induce a $K_4$ subgraph of $G$ 
      (Pl\"ucker relations).
    \end{enumerate}
    
    The relations of $\mathcal{R}(\init_{lex}(J_G))$ are generated by:
    \begin{enumerate}
        \Item 
        \[
            {x_iy_jT_{i'j'}}-x_{i'}y_{j'}T_{ij}
       \]
        for $i>i', j>j'$, $\{i,j\}\in\Delta_a$, $\{i',j'\}\in \Delta_b$ with $a<b$  {\em (Koszul type relations)};

        \Item  \begin{eqnarray}
            {x_iT_{jk}}-x_jT_{ik}   \text{ and } {y_jT_{ik}}-y_kT_{ij} \label{EN-type}
        \end{eqnarray} for $\{i<j<k\}$  an induced $C_3$ in $G$ {\em (Eagon-Northcott type relations)};
 
        \Item \[
            {T_{ij}T_{j'i'}}-T_{\min}T_{\max}
       \] for $i\leq j$ and $i'\geq j'$, $\min=(\min\{i,i'\}, \min\{j,j'\})$, $\max=(\max\{i,i'\}, \max\{j,j'\})$  
        {\em (Pl\"ucker type relations)}.
    \end{enumerate}
\end{thm}

In the rest of the paper we will focus on the case when the Pl\"ucker and Pl\"ucker type relations given in \Cref{thm: rees} are absent. In this case the ideal $J_G$ is said to be of linear type.  

Also useful to our analysis are two monomial orders on $S=R[T_{ij}: (i,j)\in E(G)]$. When analyzing relations of $\R(\init_{\lex(J_G)})$ we will use the reverse lexicographic order on $S$ induced by the following ordering of the variables
 \begin{equation}\label{eq:revlex}
 \begin{aligned}
 &T_{ij} > T_{k\ell}  \text{ if and only if } i<j \text{ or } i=j \text{ and } k<\ell\\
 &T_{ij}> x_a, y_b  \text{ for all } i,j,a,b\\
 &x_1>x_2 > \cdots x_n> y_1> y_2> \cdots >y_n.
\end{aligned}
\end{equation}
It is proven in \cite[Theorem 3.14]{ALL} that there is a {\em different} monomial ordering $\prec$ on $S$ so that the leading monomials of the relations of $\mathcal{R}(J_G)$ are the same as those of the corresponding relations for $\R(\init_{\lex}(J_G))$.
Thus we always use $\prec$ when analyzing relations of $\mathcal{R}(J_G)$.

To illustrate our work we will use the following recurring example.
\begin{ex}\label{fav-exam}
    \(G=\{\{0,1\},\{0,2\},\{1,2\},\{2,3\},\{3,4\},\{3,5\},\{4,5\}\}\)
    \begin{center}
        \begin{tikzpicture}[node distance=2cm]
            \node (0) at (0,0) [circle,fill,inner sep=2pt,label=above:1] {};
            \node (1) at (.5,.5) [circle,fill,inner sep=2pt,label=above:0] {};
            \node (2) at (1,0) [circle,fill,inner sep=2pt,label=above:2] {};
            \node (3) at (2,0) [circle,fill,inner sep=2pt,label=above:3] {};
            \node (4) at (2.5,.5) [circle,fill,inner sep=2pt,label=above:4] {};
            \node (5) at (3,0) [circle,fill,inner sep=2pt,label=above:5] {};
            \draw[-] 
                (0) -- (1)
                (0) -- (2)
                (1) -- (2)
                (2) -- (3)
                (3) -- (4)
                (3) -- (5)
                (4) -- (5);
        \end{tikzpicture}
    \end{center}
    The binomial edge ideal of \(G\) is generated by binomials as follows
    
    \[J_{G}=(f_{01}, f_{02}, f_{12},f_{23}, f_{34}, f_{45}, f_{35}).\]
    The maximal cliques are 
    \[\Delta_1=\{0,1,2\},\quad\Delta_2=\{2,3\}, \quad\Delta_3 = \{3,4,5\}.\]
    Therefore, by \Cref{thm: rees}, the relations of \(\mathcal{R}(J_{G})\) are generated by:
    
    \begin{enumerate}
        \item (Koszul relations). \(f_{01}T_{23}-f_{23}T_{01}
        ,\; f_{01}T_{34}-f_{34}T_{01},\; f_{01}T_{35}-f_{35}T_{01},\; f_{01}T_{45}-f_{45}T_{01},\; f_{02}T_{23}-f_{23}T_{02},\; f_{02}T_{34}-f_{34}T_{02},\; f_{02}T_{35}-f_{35}T_{02},\; f_{02}T_{45}-f_{45}T_{02},\; f_{12}T_{23}-f_{23}T_{12},\; f_{12}T_{34}-f_{34}T_{12},\; f_{12}T_{35}-f_{35}T_{12},\; f_{12}T_{45}-f_{45}T_{12},\; f_{23}T_{34}-f_{34}T_{23},\; f_{23}T_{35}-f_{35}T_{23},\; f_{23}T_{45}-f_{45}T_{23}.\)
        \item (Eagon-Northcott relations). \(\underline{y_1T_{02}}-y_2T_{01}-y_0T_{12},\; \underline{x_0T_{12}}-x_1T_{02}+x_2T_{01},\; \underline{y_4T_{35}}-y_5T_{34}-y_3T_{45},\; \underline{x_3T_{45}}-x_4T_{35}+x_5T_{34}.\)
        \item (Plücker relations). None.
    \end{enumerate}

    The relations of \(\mathcal{R}(\init_{lex}(J_{G}))\) are generated by:
    \begin{enumerate}
        \item[(1')] (Koszul type relations). 
        \(x_3y_2T_{01}-x_1y_0T_{23},\;x_4y_3T_{01}-x_1y_0T_{34},\;x_5y_3T_{01}-x_1y_0T_{35},\;x_5y_4T_{01}-x_1y_0T_{45},\;x_3y_2T_{02}-x_2y_0T_{23},\;x_4y_3T_{02}-x_2y_0T_{34},\;x_5y_3T_{02}-x_2y_0T_{35},\;x_5y_4T_{02}-x_2y_0T_{45},\;x_3y_2T_{12}-x_2y_1T_{23},\;x_4y_3T_{12}-x_2y_1T_{34},\;x_5y_3T_{12}-x_2y_1T_{35},\;x_5y_4T_{12}-x_2y_1T_{45},\;x_4y_3T_{23}-x_3y_2T_{34},\;x_5y_3T_{23}-x_3y_2T_{35},\;x_5y_4T_{23}-x_3y_2T_{45}.\)
        \item[(2')] (Eagon-Northcott relations).
        \(\underline{y_0T_{12}}-y_1T_{02},\; \underline{x_0T_{12}}-x_1T_{02},\; \underline{y_3T_{45}}-y_4T_{35},\; \underline{x_3T_{45}}-x_4T_{35}.\) 
        \item[(3')] (Plücker relations). None.
    \end{enumerate}  
    
We underline the leading monomials for the Eagon-Northcott and Eagon-Northcott type relations with respect to the two monomial orders discussed above.
Notice that the underlined leading terms are pairwise coprime for each of the two sets of relations. We will show that this is the case for any closed graph in \Cref{lem: coprime}. 
\end{ex}

\section{Main Result}

In this section, we prove formulae for the linear strand of $J_G^m$. 
We will extract this information from the Koszul complex on the sequence of Eagon-Northcott relations of $J_G$. We illustrate this in a concrete example.

\begin{ex}
We return to the graph from \Cref{fav-exam}.

     \begin{center}    
        \begin{tikzpicture}[node distance=2cm]
            \node (0) at (0,0) [circle,fill,inner sep=2pt,label=above:1] {};
            \node (1) at (.5,.5) [circle,fill,inner sep=2pt,label=above:0] {};
            \node (2) at (1,0) [circle,fill,inner sep=2pt,label=above:2] {};
            \node (3) at (2,0) [circle,fill,inner sep=2pt,label=above:3] {};
            \node (4) at (2.5,.5) [circle,fill,inner sep=2pt,label=above:4] {};
            \node (5) at (3,0) [circle,fill,inner sep=2pt,label=above:5] {};
            \draw[-] 
                (0) -- (1)
                (0) -- (2)
                (1) -- (2)
                (2) -- (3)
                (3) -- (4)
                (3) -- (5)
                (4) -- (5);
        \end{tikzpicture}
    \end{center}
    
For $m = 2$, we explain how the formulas in \Cref{thm: rees} arise combinatorially. Let $t=2$ denote the number of $3$-cycles and $e=7$ the number of edges of $G$. Then  \Cref{thm: rees} yields the following Betti table 
   \begin{center}
    \(e=7,\qquad t=2\)\\
    \vspace{1em}
        \begin{tabular}{c|ccccc}{\(J_{G}^2\)}
            & 0 & 1 & 2 & 3 
            \\
            \hline
            4 & \({{2 +1} \choose 2}\) & \(2\cdot 7 \cdot 2\) & \({2\cdot 2\choose 2}\)& . \\
            5 & . & \(\beta_{1,6}\) & \(\beta_{2,7}\)& \(\beta_{3,8}\)\\
            6 & . & . & \(\beta_{2,8}\) & \(\beta_{3,9}\) .
        \end{tabular}
    \end{center}

The number  in the first column of the Betti table of $J_G^2$ is the number of minimal generators of $J_G^2$. The minimal generators are the $e+1\choose 2$, pairwise products of the $e$ generators of $J_G$, which are linearly independent because for a $K_4$-free graph there are no Pl\"ucker relations.

We will see in \Cref{lem: crucial iso} that the first syzygies of  $J_G^2$  correspond to relations of $\mathcal R(J_G)$ that are quadratic in the $T_{ij}$ variables. In particular, by \Cref{thm: rees} and degree reasoning, the first linear syzygies correspond to products of Eagon-Northcott relations with  variables $T_{ij}$ of the Rees ring \eqref{eq: phi}. The top entry in the second column reflects this correspondence as every $3$-cycle in $G$ yields two Eagon-Northcott relations, one containing $x$ terms, and one containing $y$ terms. Since there are $2t$ Eagon-Northcott relations and $e$ variables $T_{ij}$ we obtain $2et$ products. In \Cref{thm: rees} we interpret these as part of the linear strand of a Koszul complex on the Eagon-Northcott relations. Similarly, we interpret the number in the last column to count the $\binom{2t}{2}$ Koszul syzygies among the $2t$ Eagon-Northcott relations.
\end{ex}

\subsection{Syzygies via relations of the Rees algebra}
In this section, we consider a more general setup than in the rest of the paper. Later, we will apply these results to the binomial edge ideals of closed graphs and to their initial binomial edge ideals. 

Let $R$ be a grade $\mathbb{K}$-algebra and let $I = (f_1, \ldots, f_s)$ be a homogeneous ideal
of $R$. Consider the $R$-algebra map $\varphi: R[T_1, \ldots, T_s] \to \R(I)$ given by $\varphi(r)=r$ for $r\in r$ and $\varphi(T_i)=f_i$. We view $R[T_1, \ldots, T_s]$ as a bigraded ring with $\deg(r)=(0,\deg_R(r))$ and $\deg(T_i)=(1,\deg(f_i))$.

We denote \(L=\ker(\varphi)\) and define the following subsets of \(L\)
\begin{eqnarray*}
L_{(a,b)} &=&\{F\in L: \deg(F)=(a,b)\},\\
L_{(a,\geq b)} &=&\{F\in L: \deg(F)=(a,c) \text{ for some } c\geq b\},\\
L_{(a,-)} &=& L_{(a, \geq 0)}.
\end{eqnarray*}
Observe that the first set is a \(\mathbb{K}\)-vector space and the last two are  \(R-\)modules. 
To give a more explicit description of these sets we utilize the notation $|\alpha|=\sum_{i=1}^s \alpha_i$ for a vector $\alpha=(\alpha_1, \ldots, \alpha_s)\in \mathbb{Z}^s$. 
We have
\begin{eqnarray}
L_{(m,-)}    & =&\{F\in R[T_1, \ldots, T_s]:\varphi(F)=0, \deg{F}=(m,-)\} \nonumber\\
    & = &\left\{F=\sum_{\alpha\in\mathbb{Z}_{\geq 0}^s, |\alpha|=m} r_\alpha T_{i}^{\alpha_{i}}\in R[T_1,\ldots, T_s] : \sum_{\alpha\in\mathbb{Z}_{\geq 0}^s, |\alpha|=m}r_\alpha f_{i}^{\alpha_{i}}=0\right\}. \label{eq: L}
\end{eqnarray}

We will approach the syzygies of powers of $I$ via the Rees algebra by means of the following lemma.

\begin{prop} \label{lem: crucial iso}
    Let $R$ be a graded $\mathbb{K}$-algebra  with $R_0=\mathbb{K}$ and let $I$ be a homogeneous ideal of $R$ with minimal homogeneous generating set $\{f_1, \dots, f_s\}$ of elements of degree $d$.
    Then there is an isomorphism between \(\Omega_1(I^m)\) and \(N =L_{(m,-)}/L_{(m,\geq md)}\) as graded $R$-modules.  
\end{prop}
    
\begin{proof}
    Let \(\{h_1, ..., h_t\}\subset\{ f^\alpha:|\alpha|=m \}\) be a minimal generating set of \(I^m\) as a \(\mathbb{K}\)-vector space. Considering a basis $e_1, \ldots, e_t$ for the free module $R^t$ and elements $r_1,\ldots, r_t\in R$ one has an $R$-module homomorphism
    \[\partial:R^t\rightarrow R 
    , \quad \partial \left(\sum_{i=1}^t r_ie_i\right)= \sum_{i=1}^t  r_ih_i.
    \] 
    Recall that \(\Omega_1(I^m)=\ker(\partial)\).

    For each $1\leq i\leq t$, fix  $\alpha^{(i)}=(\alpha^{(i)}_1, \ldots, \alpha^{(i)}_s)\in \mathbb{Z}_{\geq 0}^s$ such that the minimal generator $h_i$ decomposes as \(h_i=f_1^{\alpha^{(i)}_1}\cdots f_s^{\alpha^{(i)}_s}\).
    Consider the following map
    \begin{equation}\label{eq: defg0}
    g_0:\Omega_1(I^m)\rightarrow R[T], \quad g_0\left(\sum_{i=1}^t r_ie_i\right)=\sum_{i=1}^t r_iT_1^{\alpha_1^{(i)}}\cdots T_s^{\alpha_s^{(i)}}.
    \end{equation}
    
    Applying \(\varphi\) to an element in the image of $g_0$ yields 
    \begin{equation*}
        \varphi\left(\sum_{i=1}^t r_iT_1^{\alpha_1^{(i)}}\cdots T_s^{\alpha_s^{(i)}}\right)
        = T^m\sum_{i=1}^t r_ih_i
        = T^m \partial\left(\sum_{i=1}^t r_ie_i\right)= 0.
    \end{equation*}
    Thus, the image of $g_0$ is contained in $L_{(m, -)}$ by \eqref{eq: L}. 
    The map \(g_0\) induces an $R$-module homomorphism \(\Omega_1(I^m)\xrightarrow{g_1} N\) which maps elements of degree $v$ in the source to elements of degree $(m,v)$ in the target, thus preserving $R$-degrees.

 Next we define an inverse mapping. 
  Recall from \eqref{eq: L} that
    \begin{align}\label{eq:observe}
\sum_{\,|\alpha^{(i)}\,|=m} r_iT^{\alpha^{(i)}} = &\sum_{\,|\alpha^{(i)}\,|=m} r_iT_1^{\alpha_1^{(i)}}\cdots T_s^{\alpha_s^{(i)}}\in L_{(m,-)} 
        \iff& \sum_{\,|\alpha^{(i)}\,|=m} r_i f_1^{\alpha_1^{(i)}}\cdots f_s^{\alpha_s^{(i)}}=0.
    \end{align}
Since \(h_1,...,h_t\) generate  \(I^m\) and \(f_1^{\alpha_1^{(i)}}\cdots f_s^{\alpha_s^{(i)}}\in I^m\) whenever $|\alpha^{(i)}\,|=m$,  we can write   
\begin{equation}\label{eq:aijdef}
f_1^{\alpha_1^{(i)}}\cdots f_s^{\alpha_s^{(i)}}=\sum_{j=1}^t a_{ij}h_{j},
\end{equation}
for some \(a_{ij}\in\mathbb{K}\). 
    Then equation \eqref{eq:observe} becomes
    \begin{equation}\label{eq: becomes}
        \sum_{\,|\alpha^{(i)}\,|=m} r_iT^{\alpha^{(i)}} \in L_{(m,-)}\iff \sum_{\,|\alpha^{(i)}\,|=m} r_i \left(\sum_{j=1}^ta_{ij}h_{j}\right)=0.
        \end{equation}
    Using notation introduced in \eqref{eq:aijdef}, we define
  
    \begin{equation}\label{eq: defP0}
        p_0:L_{(m,-)}\rightarrow \Omega_1\left(I^m\right),\quad
        p_0\left(\sum_{|\alpha^{(i)}|=m} r_i T_1^{\alpha_1^{(i)}}\cdots T_s^{\alpha_s^{(i)}}\right)=\sum_{\,|\alpha^{(i)}\,|=m} r_i \left(\sum_{j=1}^t a_{ij} e_{i}\right).
    \end{equation}
    By \eqref{eq: becomes} it follows that
    \begin{align*}
        \partial\left( \sum_{\,|\alpha^{(i)}\,|=m} r_i \sum_{j=1}^t a_{ij} e_{j}\right)
        =&\sum_{\,|\alpha^{(i)}\,|=m} r_i \left(\sum_{j=1}^t a_{ij} h_{j}\right)=0.
    \end{align*}
    Thus the image of $p_0$ is contained in the syzygy module \(\Omega^1(I^m)\).

    We now wish to show that $p_0: L_{(m, -)} \to \Omega_1(I^m)$ induces a well defined map $\overline{p_0}: N \to \Omega_1(I^m)$. 
    It is enough to show that $L_{(m,md)}$ is contained in the kernel of $p_0$. For \(F=\sum_{|\alpha^{(i)}|=m}c_iT_1^{\alpha_1^{(i)}}\cdots T_s^{\alpha_s^{(i)}}\) we see that \(\deg{(F)}=(m,md)\) if and only if $\deg_R(c_i)=0$ if and only if $c_i \in \mathbb{K}$. 
    Then setting  \(b_j=\sum_{i} c_ia_{ij}\in\mathbb{K}\) yields
    \[p_0(F)=p_0\left(\sum_{i}c_iT_1^{\alpha_1^{(i)}}\cdots T_s^{\alpha_s^{(i)}}\right)=\sum_{i}c_i\left(\sum_j a_{ij}e_{j }\right)=\sum_{j=1}^t b_je_j.\]
   
    If $F\in L_{(m,mn)}$ then \(p_0(F)=\sum_{j=1}^t b_je_j\in\Omega_1(I^m)\), thus we have $\partial\left(\sum_{j=1}^t b_je_j \right)=\sum_{j=1}^t b_jh_j=0$.
 
    It follows by linear independence of $h_1, \ldots, h_s$ over $\mathbb{K}$ that \(b_j=0\) for all $1\leq j\leq t$. Therefore, \[p_0(F)=\sum_{j=1}^t b_je_j=0.\]
    
    We have shown that \(L_{(m,md)}\) is contained in the kernel of \(p_0\), thus  there is a well-defined $R$-module homomorphism $\overline{p_0}:N\to \Omega_1(I^m)$. 
    Lastly one can check that $p_0$ and $g_1$ are mutually inverse maps using equations \eqref{eq: defg0} and \eqref{eq: defP0}, and the fact that $g_1$ is induced by $g_0$.
\end{proof}

    While the previous result concerns only the first syzygy modules of the powers of an ideal, under additional assumptions on the relations of the Rees algebra  we can use the Koszul complex to find a formula for the entire linear strand of the powers. 

\begin{prop}\label{lem: linear strand}
    Let $I$ be a homogeneous ideal with  minimal homogeneous generating set $\{f_1, \dots, f_s\}$ of elements of degree $d$. Let $S = R[T_1, \ldots, T_s]$ and let $L\subset S$ denote the ideal of relations of the Rees algebra $\mathcal{R}(I)$. 
    If $L_{(1, d+1)}$ is spanned by a $S$-regular sequence $g_1, \dots, g_r$ and satisfies $L_{(m,md+1)}=(g_1, \dots, g_r)_{(m,md+1)}$ and $L_{(m,md)}=0$ for some $m\geq 1$, then  the Betti numbers in the linear strand of $I^m$ are given by 
    \[\beta_{(i, md+i)}(I^m) = \binom{r}{i}\binom{s+m-i-1}{m-i}.
    \]
\end{prop}

\begin{proof}
    Set $E=(g_1, \ldots, g_rs)\subset S$ and consider the Koszul complex $K_\bullet$ of  the regular sequence $g_1,\dots, g_r$. This is a minimal free resolution for $S/E$: 
    \[K_{\bullet}: 0 \to S(-m, -m(d+1))^{r \choose m} \to \ldots \to S(-i, -i(d+1))^{r \choose i} \to \ldots \to S(-1,-(d+1))^{r \choose 1} \to S \to 0.\]
Restricting bidegrees to $(m, \geq md)$,
    and augmenting by $S(m, \geq md)\to I^m \to 0$ on the right, we obtain the complex
    \begin{equation}\label{complex}
    0 \to S(-m, -m(d+1))^{r \choose m}_{(m, \geq md)} \to \ldots \to S(-i, -i(d+1))_{(m, \geq md)}^{r \choose i} \to \ldots \to S_{(m, \geq md)} \to I^m \to 0.
    \end{equation}
 Observe that in each bidegree $(u,v)$ there is a graded $\mathbb{K}$-vector space isomorphism
 \begin{equation}\label{eq: free1}
 S_{(u,v)}=\bigoplus_{1\leq i_1\leq i_2\leq \cdots \leq i_u\leq s}R_{v-du}\prod_{j=1}^u T_{i_j} \cong {R_{v-du}}^{\binom{s+u-1}{u}}.
 \end{equation}
This yields that
 $S(-i, -i(d+1))_{(m, \geq md)}$ is a free-graded $R$-module with the following decomposition:
    \begin{equation}\label{eq: free}
    \begin{split}
        S(-i, -i(d+1))_{(m, \geq md)} = S_{(m-i, \geq md-i(d+1))} 
    = S_{(m-i, \geq d(m-i)} (0,-i) \\ 
    \stackrel{\eqref{eq: free1}}{\cong} \left[ \bigoplus_{v \geq d(m-i)} {R_{v-d(m-i)}}^{\binom{s + (m-i) -1}{m-i}}\right](-i) = 
    R(-i)^{\binom{s + (m-i) -1}{m-i}}.
     \end{split}
     \end{equation}
Thus \eqref{complex} is an augmented linear complex of free $R$-modules.
    
    Note that, while it is exact in higher homological degrees, the complex \eqref{complex} is not necessarily exact at $S(-1, -n)_{(m, \geq md)}^{r \choose 1} \to S_{(m, \geq md)} \to I^m$ since we have $\ker(S_{(m, \geq md)} \to I^m) = L$, not $E$. 
However, restricting the sequence displayed above to bidegree $(m,md+1)$ we have
    \begin{eqnarray}
        \ker(S_{(m,  md+1)} \to [I^m]_{md+1})=L_{(m,md+1)} \label{eq: top 1}\\
        \Img(S(-1,-d)_{(m,  md+1)}^{\binom{r}{1}}\to S_{(m,  md+1)}) =E_{(m,md+1)} \label{eq: top 2}.
    \end{eqnarray}
Equality in \eqref{eq: top 1} follows from   \Cref{lem: crucial iso} and the hypothesis $L_{(m,md)}=0$. 
The vector spaces in \eqref{eq: top 1} and \eqref{eq: top 2}  are equal
by the hypothesis $L_{(m,md+1)}=(g_1, \dots, g_r)_{(m,md+1)}=E_{(m,md+1)}$. Since the complex \eqref{complex} is exact in degrees $md$ and $md+1$, it is isomorphic to the linear strand in the $R$-free resolution of $I^m$ by \cite[Corollary 1.2]{HKM}.
    
   Thus we obtain the formula $\beta_{i, md + i}(I^m) = {r \choose i}{s+m-i-1 \choose m-i}$ by combining \eqref{complex} and \eqref{eq: free}.
\end{proof}

\subsection{The linear strand of $J_G^m$}

If $G$ is a closed graph that is $K_4$-free, then the Rees algebra of $J_G$ does not exhibit any Pl\"ucker relations and the  Rees algebra of $\init_{\lex}(J_G)$ does not have any Pl\"ucker-type relations; see \Cref{thm: rees} for the definition of these relations which shows that they correspond to $K_4$ subgraphs. Therefore we can derive from \Cref{lem: crucial iso} a particularly simple expression for the first syzygies of powers of $J_G$.

\begin{cor}
\label{no k_4 iso}
Let $G$ be a closed graph with no induced $K_4$. Let $L$ denote the ideal of relations of the Rees algebra $\mathcal{R}(J_G)$  and let $L'$ the ideal of relations of the Rees algebra $\mathcal{R}(\init_{\lex} J_G)$. Then, with notation as in \eqref{eq: L}, for each integer $m\geq 1$ there are graded $R$-module isomorphisms 
\[
L_{(m,-)}\cong\Omega^1(J_G^m) \qquad \text{and} \qquad L'_{(m,-)}\cong\Omega^1((\init_{\lex} J_G)^m).
\]
\end{cor}
\begin{proof}
    Since the graph $G$ is closed both the ideal $J_G$ and its initial ideal $\init_{\lex}(J_G)$ are generated in degree two. If there are no induced $K_4$ graphs, then $L_{(m,2m)}=0$ and by \Cref{lem: crucial iso} it follows that $\Omega_1(J_G^m) \cong L_{(m,-)}$.  Likewise, if there are no induced $K_4$ graphs, then $L'_{(m,2m)}=0$ and by \Cref{lem: crucial iso} we have $\Omega_1(J_G^m) \cong L_{(m,-)}$. 
\end{proof}

Recall the definitions of Eagon-Northcott  and Eagon-Northcott type relations from \eqref{EN-rels} and \eqref{EN-type}, respectively. With an eye towards applying \Cref{lem: linear strand}, we show that the Eagon-Northcott-type and, separately, the  Eagon-Northcott relations form a regular sequence in $S$.

\begin{lemma}\label{lem: coprime}
Let $G$ be a closed graph equipped with an interval labeling. Assume that $G$ has no cliques of size four. Then
\begin{enumerate}[nosep]
\item the set of Eagon-Northcott type relations \eqref{EN-rels} form a regular sequence.
\item the set of Eagon-Northcott relations \eqref{EN-type} form a regular sequence.
\end{enumerate}
\end{lemma}
\begin{proof}
To prove (1) and (2)  it suffices to show that the leading monomials of each set of relations form a regular sequence \cite{reg_sequence_ref}. Since the leading monomials are the same by the remarks following \Cref{thm: rees}, 
we  are able to prove (1) and (2) simultaneously.

The set of leading monomials of the  Eagon-Northcott type and Eagon-Northcott \eqref{EN-rels}  relations with respect to the monomial order discussed in \eqref{eq:revlex} and the order $\prec$ described after this equation, respectively,
is  
\[
\mathcal{L}=\left \{x_iT_{jk}, y_j T_{ik} \mid \{i<j<k\} \text{ is a clique  in } G \right \}.
\]
We now show that the monomials listed in $\mathcal{L}$ are pairwise coprime, which is equivalent to them forming a regular sequence.

In our proof we shall employ an alternate characterization of closed graphs from \cite[Proposition 4.8]{CrupiRinaldo}: $G$ is closed if and only if for all integers $1 \leq i < j < k \leq$ n, if $\{i,k\}\in E(G)$, then $\{i,j\}\in E(G)$ and $\{j,k\}\in E(G)$. We refer to this property as the closed labeling condition.

Consider the following cases:
\begin{itemize}[nosep]
\item[{\em Case 1:}] If $x_iT_{jk} \in \mathcal{L}$ and  $x_i'T_{j'k'} \in \mathcal{L}$ are distinct monomials, then  $\{i<j<k\}$ and  $\{i'<j'<k'\}$ are cliques in $G$. If $i=i'$ then by the definition of closed graphs $\{j,j'\}, \{j,k'\}, \{j',k\}, \{k,k'\}\in E(G)$, so $\{i,j,k,j',k'\}$ is a clique in $G$ of size at least four. If $j=j'$ and $k=k'$ then  because $G$ is closed $\{i,i'\}\in E(G)$, so $\{i,j,k,i'\}$ is a clique in $G$ of size four. Since both possibilities contradict the hypothesis, we conclude that $x_iT_{jk}$ and  $x_i'T_{j'k'}$ are coprime.
\item[{\em Case 2:}] If $x_iT_{jk} \in \mathcal{L}$ and  $y_{j'}T_{i'k'} \in \mathcal{L}$ are distinct monomials, then  $\{i<j<k\}$ and  $\{i'<j'<k'\}$ are cliques in $G$. If $j=i'$ and $k=k'$ then we have $i<j=i'<j'< k=k'$. By the closed labeling condition since $\{i,k\}$ is an edge of $G$, $\{i,j,j',k\}$ is a clique in $G$ of size four.  Since this contradicts the hypothesis, we conclude that $x_iT_{jk}$ and  $y_{j'}T_{i'k'}$ are coprime.
\item[{\em Case 3:}]  If $y_j T_{ik}  \in \mathcal{L}$ and  $y_{j'}T_{i'k'} \in \mathcal{L}$ are distinct monomials, then  $\{i<j<k\}$ and  $\{i'<j'<k'\}$ are cliques in $G$. If $j=j'$ then because $G$ is closed at least one of the pairs $\{i,i'\}$ and $\{k, k'\}$ is distinct and is an edge of $G$. If $i\neq i'$ and $\{i,i'\}\in E(G)$, assuming without loss of generality that $i<i'$, the closed labeling condition applied to $i<i'<k$ yields that these three vertices form a clique and thus $\{i,i',j,k\}$ is a clique in $G$ of size four. If $i=i'$ and $k\neq k'$ with $\{k,k'\}\in E(G)$, then $\{i,j,k,k'\}$ is a clique of size four. 
If $i=i'$ and $k=k'$ then we have $i=i'<j, j'< k=k'$ and by the closed labeling condition it follows that $\{i,j,j',k\}$ is a clique in $G$ of size four. Since every possibility contradicts the hypothesis, we conclude that $x_iT_{jk}$ and  $y_j'T_{i'k'}$ are coprime.
\end{itemize}
In view of the arguments preceding the case analysis, this concludes the proof.
\end{proof}

We can apply this result to the binomial edge ideal of a $K_4$-free graph to obtian our main result.
\begin{thm}\label{thm: linear strand}
    Let $G$ be a closed $K_4$-free graph with $e$ edges and $t$ triangles (3-cycles). The Betti numbers in the $2m$-linear strands of $J_G^m$ and $\init_{\lex}(J_G)^m$are given by
    \[
    \beta_{i, 2m+i}(J_G^m)=\beta_{i, 2m+i}(\init_{\lex}(J_G)^m)=\binom{e+m-i-1}{m-i}\binom{2t}{i}.
    \]
\end{thm}

\begin{proof}
    Fix an integer $m\geq 1$ and let $L$ and $L'$ denote the set of relations of $\R(J_G)$ and of $\R(\init_{\lex}(J_G))$, respectively. The assumption that $G$ is $K_4$-free ensures that there are no Pl\"ucker or Pl\"ucker type relations, therefore by \Cref{thm: rees} we have  $L_{(m,2m)}=L'_{(m,2m)}=0$.
    
    By the \Cref{lem: coprime}, the set of Eagon-Northcott relations and the set of Eagon-Northcott type relations given in \Cref{thm: rees} form a regular sequence.  Moreover, 
    for degree reasons based on \Cref{thm: rees}, every element of $L_{(m,2m+1)}$ must belong to  the ideal generated by the Eagon-Northcott relations and every element of $L'_{(m,2m+1)}$ must belong to the ideal generated by the  Eagon-Northcott type relations, respectively. We can then apply \Cref{lem: linear strand} both for $J_G$ and for $\init_{\lex}(J_G)$ with $r=2t$ (the number of Eagon-Northcott or Eagon-Northcott type relations) and $s=e$ (the number of minimal generators of $J_G$ and $\init_{\lex}(J_G)$) to obtain the claimed formulas. The hypothesis that $G$ is closed is used to ensure that the number of minimal generators of  $\init_{\lex}(J_G)$ is the same as that of $J_G$; see \cite[Theorem 7.2]{HHO}.
\end{proof}

\section{Acknowledgements}
This research was funded by NSF RTG Grant DMS-2342256 \emph{Commutative Algebra at Nebraska} and by SECIHTI grants CF-2023-G-33 \emph{Redefiniendo fronteras entre el álgebra conmutativa, la teoría de códigos y la teoría de singularidades} and Grant CBF 2023-2024-224: \emph{Álgebra conmutativa, singularidades y códigos}, as part of the the International REU in Commutative Algebra, held at CIMAT in Summer 2025, and organized by CIMAT and the University of Nebraska-Lincoln.

We would also like to thank Alexandra Seceleanu, Shahriyar Roshan Zamir, Eloísa Grifo, Luis Núñez-Betancourt, and Jack Jeffries for their valuable discussions and mentorship.

\bibliography{references}
\bibliographystyle{plain}

\end{document}